\documentclass[12pt, pdflatex]{amsart}

\usepackage[alphabetic,nobysame]{amsrefs}
\usepackage{amssymb}
\usepackage{fancyhdr}
\usepackage{bbm}
\usepackage[all]{xy}
\usepackage{xcolor}
\usepackage{hyperref}
\hypersetup{
  colorlinks=true,
  linkcolor=blue,
  citecolor=blue
}
\usepackage{cleveref}% Has to be loaded after hyperref
\usepackage{mathrsfs}
\usepackage{mleftright}
\usepackage{mathdots}

\newtheorem{theorem}{Theorem}
\newtheorem*{theorem*}{Theorem}
\newtheorem{lemma}[theorem]{Lemma}
\newtheorem*{lemma*}{Lemma}
\newtheorem{proposition}[theorem]{Proposition}
\newtheorem{claim}[theorem]{Claim}
\newtheorem{corollary}[theorem]{Corollary}

\newtheorem{itheorem}{Theorem} % for introduction

\theoremstyle{definition}

\newtheorem{remark}[theorem]{Remark}
\newtheorem*{definition*}{Definition}
\usepackage{graphicx}

\numberwithin{equation}{section}

\numberwithin{theorem}{section}

\newcommand{\RR}{\mathbf{R}}
\newcommand{\NN}{\mathbf{N}}

\newcommand{\ZZ}{\mathbf{Z}}

\newcommand{\sP}{\mathscr{P}}
\newcommand{\sPf}{\mathscr{P}_\mathrm{fin}}
\newcommand{\sPt}{\mathscr{P}_\mathrm{pin}}

\newcommand{\incr}{\NN_\nearrow}
\newcommand{\incrR}{\RR_\nearrow}

\def\dd{\mathrm{d}}

\DeclareMathOperator{\End}{End}
\DeclareMathOperator{\Emb}{Emb}
\DeclareMathOperator{\Epi}{Epi}

\begin{document}

\title[Fixed points, faces and deleting entries]{A fixed-point theorem for face maps,\\ or deletion-tolerant  random finite sets}

\author[Hutchcroft]{Tom Hutchcroft}
\author[Monod]{Nicolas Monod}
\author[Tamuz]{Omer Tamuz}
%    Address of record for the research reported here
\address{T.H., O.T.: Caltech, Pasadena, CA 91125, USA}
\address{N.M.: EPFL, 1015 Lausanne, Switzerland}

%    Information for third author

%\thanks{}
%\date{\today}

\begin{abstract}
We establish a fixed-point theorem for the face maps that consist in deleting the $i$th entry of an ordered set. Furthermore, we show that there exists random finite sets of integers that are almost invariant under such deletions. Consequences for various monoids of order-preserving transformations of $\NN$ are discussed in an appendix.
\end{abstract}

\maketitle
\section{Introduction}
\subsection*{Results}
Consider a finite set $E$ of integers. If $E$ has at least $k$ elements, then we can study the operation $\alpha_1$ of deleting its smallest element, $\alpha_2$ of deleting the second smallest, and so on up to $\alpha_k$. More formally, if $E=\{y_1 < \cdots < y_{n}\}$ then $\alpha_i (E) = E \smallsetminus \{y_i\}$.

Our first result is that there exist random sets which are almost invariant under all these deletion operations.

\begin{itheorem}\label{thm:random-sets}
Fix $k\geq 2$ and $\epsilon >0$.

There exists a finitely supported probability distribution $\mu$ on the collection of finite sets $E\subseteq \NN$ with at least $k$ elements such that $\alpha_i (\mu)$ is $\epsilon$-close to $\mu$ in total variation distance for all $1\leq i \leq k$.
\end{itheorem}

Working with integers is merely a matter of convenience since $\alpha_i$ is defined in terms of the ordering only. Thus, the above result holds unchanged when $\NN$ is replaced by %any totally ordered set without upper bound, or
any partially ordered set containing chains of arbitrary length. %This can be conceptualized by viewing the $\alpha_i$ as face maps in the sense of algebraic topology.

\bigskip

Our next result is a fixed-point principle for continuous affine transformations of compact convex sets (in any Hausdorff topological vector space). Recall that this is the context of the Markov--Kakutani Theorem, which states that any \emph{commuting} family of such transformations has a common fixed point. This fails of course in general for non-commuting transformations.

\begin{itheorem}\label{thm:fixed-pt}
Let $s_i$ be a family of continuous affine transformations of a compact convex set $K\neq\varnothing$, where $i\in\NN$.

\nopagebreak
If $s_j s_i = s_i s_{j+1}$ holds for all $1\leq i \leq j$, then $K$ contains a point fixed by all $s_i$.
\end{itheorem}

It is imperative in \Cref{thm:fixed-pt} that the family $s_i$ be infinite: if for some $n\geq 2$ we had only $s_1, \ldots, s_n$, then the relations  $s_j s_i = s_i s_{j+1}$ would only make sense for $j\leq n-1$. Any $K$ that is not reduced to a point admits such transformations without a common fixed point. Indeed, choose $x\neq y$ in $K$ and define $s_i$ to be the constant map to $x$ if $1\leq i \leq n-1$, whilst $s_n$ is the constant map to $y$.

\subsection*{Discussion of the results}
The relations $s_j s_i = s_i s_{j+1}$ are very familiar in algebraic topology, where they appear as face relations, as formalised notably by Eilenberg--Zilber~\cite{Eilenberg-Zilber50}. Face operations are usually defined separately in every dimension; for instance, with $n$ fixed, the operations $\alpha_i$ above satisfy the facial relations for  $i\leq j\leq n-1$ and this is a standard model for faces. In that context, \Cref{thm:random-sets} states that any infinite-dimensional simplical complex carries random simplices almost-invariant under passing to faces.

Instead, we propose to define a single monoid $S$ given by generators $s_i$ subject to the above relations. Then \Cref{thm:fixed-pt} states that this monoid is \textbf{amenable}. This definition of amenability in terms of fixed points is equivalent to the existence of a left-invariant mean, as studied extensively since the foundational articles by von Neumann~\cite{vonNeumann29} and Day~\cite{Day57}. We refer to \Cref{sec:amen} below for terminological clarifications. In particular, we record there that $S$ satisfies one form of F\o lner's condition but not another one.

Observe that $S$ is the increasing union of the sub-monoids $S_n<S$ generated by the first $n$ elements $s_1, \ldots , s_n$. We shall prove that $S_n$ coincides with the monoid defined by the same presentation as $S$ but truncated on these $n$ generators. It then follows from the above remark on fixed-point-free actions that none of the $S_n$ is amenable when $n\geq 2$, see \Cref{prop:Sn}.

That an amenable monoid $S$ is the nested union of non-amenable sub-monoids $S_n$ cannot happen for groups. Is is not especially remarkable in the world of monoids and semigroups; nonetheless, it illustrates that the monoid $S$ is not too elementary and that its amenability ``hides at infinity''.

The monoid $S$ might deserve further interest. It acts naturally on many sequence spaces; in \Cref{sec:Finetti}, we investigate some of the associated ergodic theory, showing notably that the $S$-action on shift spaces exhibits a De Finetti type rigidity of invariant measures. The relation between $S$ and sequences is further described in \Cref{sec:related}, where we broaden the viewpoint: both $S$ and its opposite monoid sit in larger Polish monoids. We clarify which ones are amenable, be it as topological or abstract monoids.

\medskip
Our approach to the above two theorems rests on the observation that they are two visages of the same phenomenon. In particular, an explicit construction of random finite sets would lead to a proof of the fixed-point statement. However, it turns out that our argument ends up instead using a fixed-point principle to establish the existence of the random sets.

In fact it seems non-trivial to control random finite sets directly; we could only verify that explicit constructions will satisfy \Cref{thm:random-sets} for $k=2$. In \Cref{sec:explicit}, we exhibit a concrete model of random sets for which one could conjecture almost-invariance.

\medskip
In closing, we cannot evade a mention of Thompson's group, which admits the group presentation
\begin{equation*}
F = \big\langle g_i : g_j g_i = g_i g_{j+1} \ \forall\, 1\leq i < j  \big\rangle.
\end{equation*}
Consider the monoiod $F_+$ given by that same presentation, but as monoid presentation. Then $S$ is, by definition, a monoid quotient of $F_+$. It is well-known that $F$ is amenable if and only $F_+$ is so; this follows immediately from a normal form decomposition for $F$. On the other hand, the (non-)amenability of $F$ remains a notorious open problem.

A step of our proof of \Cref{thm:fixed-pt} can evoke echoes of the work of Moore for $F$, specifically~\cites{Moore_hindman,Moore_idem}. It is not clear, however, that there exists a formal connection there.

Thus, one of the appeals of $S$ is that it is a complicated quotient of $F_+$; by contrast, we recall that the group $F$ admits no proper quotients at all, except via its abelianization $\ZZ^2$.

In that context, an indication of the richness of $S$ and a further motivation for \Cref{thm:fixed-pt} is that $F_+$ embeds into a semi-direct product of $S$ by an abelian monoid (or group), which is amenable. See the end of \Cref{sec:amen}.

\subsection*{Convention}
In this text, the natural numbers $\NN=\{1, 2, \ldots \}$ start with $1$. We write $\sP_k(\NN)$ for the collection of subsets of $\NN$ of size $k$, and $\sP_{\geq k}(\NN)$ for those of size at least $k$.

\section{Proof of Theorem~\ref{thm:random-sets}: A law of very large numbers}\label{sec:idem}
A \textbf{mean} on a set $X$ is a positive normalized linear functional on $\ell^\infty(X)$, or equivalently a finitely additive probability measure on $X$. For instance, probability distributions on $X$ are means, and are in fact dense in the space of means for the weak-* topology of duality with $\ell^\infty(X)$ by Goldstine's theorem~\cite{Goldstine38}. We can also consider means as Radon probability measures on the Stone--\v{C}ech compactification of $X$.

\begin{remark}\label{rem:subset}
Occasionally, we will use implicitly the fact that means on a subset $Y\subseteq X$ correspond canonically and uniquely to means $\nu$ on $X$ satisfying $\nu(Y)=1$.
\end{remark}

A mean $\nu$ is called \textbf{diffuse} if $\nu(E)=0$ for every finite set $E$. Informally, we think of a diffuse mean on $\NN$ as representing random very large numbers.

\bigskip

We write $\incr^k$ for the set of strictly increasing sequences of $k$ positive integers; these are naturally identified with the subsets of $\NN$ of size $k$. Given $k\geq 1$ and $1\leq j \leq k+1$, we consider the operation $\alpha_j\colon \incr^{k+1} \to \incr^{k}$ which deletes the $j$th element of the sequence. For instance, given $y_1 < \cdots  < y_{k+1}$, we have 
\begin{align*}
\alpha_1(y_1, y_2,y_3 \ldots , y_{k+1}) &= (y_2,y_3,  \ldots , y_{k+1}), \text{ and}\\
\alpha_2(y_1, y_2,y_3 \ldots , y_{k+1}) &= (y_1,y_3,  \ldots , y_{k+1}).
\end{align*}
These are the same maps $\alpha_i$ defined before the statement of Theorem~\ref{thm:random-sets} after identifying $\incr^k$ with $\sP_k(\NN)$.

Given a mean $\nu$ on $\NN$, we define a mean $\nu_k$ on  $\incr^k$ recursively. Start with $\nu_1 = \nu$ on $\incr^1 = \NN$. Then, for $k\geq 1$ and $f\in \ell^\infty(\incr^{k+1})$ we set
\begin{equation*}
\nu_{k+1}(f) = \nu \mleft( x \mapsto \nu_{k} \mleft( f^x \mright) \mright)
\end{equation*}
wherein $f^x \in \ell^\infty(\incr^{k})$ is defined by
\begin{equation*}
f^x (y_1, \ldots , y_{k}) = f (y_1, \ldots , y_{k},  y_{k}+x).
\end{equation*}
\begin{remark}\label{rem:RW}
We can loosely think of $\nu_k$ as recording successive sums of IID copies of $\nu$, or the trajectory of a $\nu$-random walk on $\NN$. Specifically, if $A_j\subseteq \NN$ are any sets and $A\subseteq \incr^{k+1}$ denotes the collection of those $(y_1, \ldots, y_{k+1})$ for which $y_1\in A_1$ and $y_{j+1} - y_j \in A_{j+1}$, then we have $\nu_{k+1}(A) = \nu(A_1) \nu(A_2) \cdots \nu(A_{k+1})$.

However we caution the reader that this is in no way a formal definition of $\nu_k$ because means on products are not determined by their values on product sets, and any form of Fubini fails. This is related to the fact that the  Stone--\v{C}ech compactification of a product is much larger than the product of the  Stone--\v{C}ech compactifications. Instead, the definition above is the non-commutative convolution in the sense of Arens~\cite{Arens51}.
\end{remark}

\bigskip

From now on we specialise to a mean $\nu$ satisfying the \textbf{idempotence} condition
\begin{equation}\label{eq:idem}
\forall \, f\in\ell^\infty(\NN) : \nu(f) = \nu \mleft( x \mapsto \nu\mleft(x'\mapsto f(x+x') \mright) \mright).
\end{equation}
This condition holds notably whenever $\nu$ is an invariant mean, since invariance implies $\nu\mleft(x'\mapsto f(x+x') \mright) = \nu(f)$. The classical source of invariant means on $\NN$ is to take any accumulation point, in the compact space of means, of the sequence of uniform distributions on $\{1, \ldots, n\}$. Seeing means as functionals on $\ell^\infty(X)$, this coincides with the sequence of Ces{\`a}ro averages.

We note that the only finitely supported convolution idempotent on $\ZZ$ is the point mass at zero; since $\nu$ is supported on $\NN=\{1,2,\ldots\}$, one readily checks that it must be diffuse.

\begin{theorem}\label{thm:omega-invariant}
For all $k\geq 1$ and  $1\leq j \leq k+1$ we have $\alpha_j \nu_{k+1} = \nu_{k}$.
\end{theorem}

\begin{remark}
In \Cref{prop:no_random_walk} we show that if $\nu$ is a \textit{countably additive} probability measure on $\NN$ and we define the random walk measures $\nu_k$ as above, then $\alpha_1\nu_{k+1}$ has total variation distance at least $1/4$ from $\nu_k$. As such, \Cref{thm:omega-invariant} establishes a property of certain ``finitely additive random walks'' that cannot hold \emph{even approximately} for true random walks. This apparent paradox is not contradictory since the Arens convolution $(\mu,\nu)\mapsto \mu*\nu$ and the operation sending $\nu$ to $\nu_k$ defined above are not continuous  on the space of means.
\end{remark}

\begin{proof}[Proof of \Cref{thm:omega-invariant}]
We first show that for every $k \geq 1$ it holds that $\alpha_{k+1}\nu_{k+1} = \nu_k$. Let $f\in\ell^\infty( \incr^k)$. We have
\begin{equation*}
(\alpha_{k+1} \nu_{k+1})(f) = \nu_{k+1} (f\circ \alpha_{k+1}) = \nu \mleft( x \mapsto \nu_{k} \mleft(  (f\circ \alpha_{k+1})^x \mright) \mright).
\end{equation*}
Observe that $(f\circ \alpha_{k+1}) ^x = f$ holds by definition; adding a new entry at the end of a sequence and then deleting it leaves a sequence unchanged. Therefore,
\begin{equation*}
(\alpha_{k+1} \nu_{k+1})(f) = \nu \mleft( x \mapsto \nu_{k} (f) \mright) = \nu_{k} (f)
\end{equation*}
as desired.

Next, we show that for every $k \geq 1$ it holds that $\alpha_{k}\nu_{k+1}=\nu_k$. Again let  $f\in\ell^\infty( \incr^k)$. In the special case $k=1$, $f\in\ell^\infty( \incr^1)$, and we note that $(f\circ \alpha_1) ^x(y_1) = f(y_1 + x)$. Therefore, using the idempotence condition~\eqref{eq:idem},
\begin{equation*}
(\alpha_1 \nu_{2})(f) = \nu \mleft( x \mapsto \nu\mleft(y_1 \mapsto  f(y_1+x) \mright)  \mright) = \nu (f) = \nu_1(f).
\end{equation*}

If $k\geq 2$, we expand twice as above:
\begin{multline*}
(\alpha_k \nu_{k+1})(f)  = \nu \mleft( x \mapsto \nu_{k} \mleft(  (f\circ \alpha_k)^x \mright) \mright)  = \\
=\nu \mleft( x \mapsto  \nu \mleft( x' \mapsto\nu_{k-1} \mleft(  ((f\circ \alpha_k)^x)^{x'} \mright)  \mright)  \mright).
\end{multline*}
Now we claim $((f\circ \alpha_k)^x)^{x'} = f^{x+x'}$. Indeed, given $y_1 < \cdots < y_{k-1}$,
\begin{multline*}
((f\circ \alpha_k)^x)^{x'} (y_1, \ldots , y_{k-1}) = (f\circ \alpha_k)^x (y_1, \ldots , y_{k-1}, y_{k-1}+x') = \\ 
= (f\circ \alpha_k)(y_1, \ldots , y_{k-1},  y_{k-1}+ x',  y_{k-1}+ x'+x ) = f(y_1, \ldots , y_{k-1},  y_{k-1}+ x'+x ).
\end{multline*}
Therefore
\begin{equation*}
(\alpha_k \nu_{k+1})(f)  = \nu \mleft( x \mapsto  \nu \mleft( x' \mapsto\nu_{k-1} \mleft(  f^{x+x'} \mright)  \mright)  \mright) 
\end{equation*}
which, using idempotence again, is
\begin{equation*}
\nu \mleft( x \mapsto\nu_{k-1} \mleft(  f^{x} \mright)  \mright) = \nu_k(f).
\end{equation*}

We have so far shown that $\alpha_j\nu_{k+1}=\nu_k$ for all $k$ and $j \in \{k,k+1\}$; in particular that completes the proof of the case $k=1$. To complete the entire proof we need to show that the same holds for any $k \geq 2$ and $j \leq k-1$. To this end, assume by induction that this holds for $k-1$, i.e., that $\alpha_{j}\nu_{k}=\nu_{k-1}$. Given $f\in\ell^\infty( \incr^k)$, the condition $j\leq k-1$ implies $(f\circ \alpha_j) ^x = f^x \circ \alpha_j$. Therefore,
\begin{equation*}
(\alpha_j \nu_{k+1})(f)   = \nu \mleft( x \mapsto \nu_{k} \mleft(  f^x\circ \alpha_{j} \mright) \mright)  = \nu \mleft( x \mapsto (\alpha_{j} \nu_{k}) \mleft(  f^x )\mright) \mright).
\end{equation*}
By the induction hypothesis, $\alpha_j\nu_k = \nu_{k-1}$, and so this is
\begin{equation*}
\nu \mleft( x \mapsto \nu_{k-1} \mleft(  f^x\mright) \mright) = \nu_k(f).
\end{equation*}
\end{proof}

\begin{proof}[Proof of \Cref{thm:random-sets}]
We continue to identify $\incr^k$ with $\sP_k(\NN)$; thus, for $i\geq k$, we can take all the above defined $\nu_i$ to be means on $\sP_{\geq k}(\NN)$. 

Fix $k \geq 2$ and consider the Ces{\`a}ro average $\bar \nu_n = \frac{1}{n}\sum_{i=k+1}^{k+n}\nu_i$. Then Theorem~\ref{thm:omega-invariant} implies that $\bar\nu_n$  is $2/n$-close to $\alpha_i(\bar\nu_n)$ for all $i \leq k$. Let $\bar\nu$ be an accumulation point of this sequence in the space of means on $\sP_{\geq k}(\NN)$. Then $\alpha_i \bar\nu = \bar\nu$ for all $1\leq i \leq k$.

To complete the proof we show that $\bar\nu$ can be approximated by finitely supported probability measures on $\sP_{\geq k}(\NN)$. This follows from standard (but non-trivial) arguments due to Day~\cites{Day57}. Indeed, by Goldstine's theorem, there is a net $\mu_q$ of finitely supported probability distributions on  $\sP_{\geq k}(\NN)$ converging weak-* to $\bar\nu$, where $q$ ranges in some directed index set. It follows that for all $1\leq i \leq k$ we have $\lim_q (\alpha_i \mu_q - \mu_q)=0$ weakly. Here we used that the weak topology on $\ell^1$ is none other than the weak-* topology when viewed in the dual of $\ell^\infty$. Now an application of the Hahn--Banach theorem known in this situation as ``Mazur's trick'' shows that we can obtain another net $\widetilde\mu_p$, possibly indexed by another set of indices $p$, by taking each $\widetilde\mu_p$ to be a convex combination of finitely many  $\mu_q$ with $q$ arbitrarily large and such that $\lim_p (\alpha_i \widetilde\mu_p - \widetilde\mu_p)=0$ in $\ell^1$-norm for each $i$. This is exactly Day's argument in~\cite[\S5]{Day57}. Recalling that the total variation distance coincides with half the $\ell^1$-norm, any $\widetilde\mu_p$ with $p$ large enough will give us the desired model of random finite sets.
\end{proof}

\begin{remark}
Note that this proof yields the existence of deletion-almost-invariant probability measures on the finite subsets of $\NN$, but does not provide a finitary construction nor any quantitative estimates. Instead, it relies heavily on abstract arguments which require (some form of) the axiom of choice. Since the convolution operation used in the definition of the $\nu_i$ is not jointly continuous, and since Day's argument is highly non-constructive, we do not see a path to directly extract a construction from this proof. In \Cref{sec:explicit} we provide a conjectural explicit construction; one should exist at any rate by some form of Shoenfield's theorem.
\end{remark}

\section{Deleting versus filling}\label{sec:up-down}
\begin{flushright}
\begin{minipage}[t]{0.9\linewidth}\itshape\small
\begin{flushright}
Alice couldn’t help smiling as she took out her memorandum-book, and worked the sum for him. Humpty Dumpty took the book, and looked at it carefully. ``That seems to be done right---'' he began. ``You're holding it upside down!'' Alice interrupted. 
\end{flushright}
\begin{flushright}
\upshape\small
Through the Looking-Glass~\cite[Chap.~VI]{Carroll1872}
\end{flushright}
\end{minipage}
 \end{flushright}
The construction of the means $\nu_k$ in \Cref{thm:omega-invariant} relies on a first reversal, since the deletion-tolerance is obtained by walking backwards with decreasing $k$ along the ``diffuse random walk''. We now need to perform a second back-flip by essentially taking complements.

Write $\sPf(\NN)$ for the set of finite subsets of $\NN$. An element $E\in \sPf(\NN)$ is \textbf{pin-headed} if it is non-empty and $\max(E) -1 \notin E$. We write $\sPt(\NN)$ for the collection of these elements and $[a,b]$ for intervals of integers.

In this section we establish a bijection between $\sPf(\NN)$ and $\sPt(\NN)$, and study the action of the facial monoid on these sets. The usefulness of this bijection will become apparent in the next section, where both sets will be shown to be identified with the facial monoid $S$.

The following is a straightforward verification based on the definition of pin-headed.
\begin{lemma}\label{lemma:delete-fill}
The map
\begin{align*}
\sPt(\NN) &\longrightarrow \sPf(\NN)\\
E &\longmapsto [1, \max(E)] \smallsetminus E
\end{align*}
is a bijection with inverse
\begin{align*}
\sPf(\NN) &\longrightarrow \sPt(\NN)\\
\varnothing  &\longmapsto \{1\}\\
\varnothing \neq F &\longmapsto [1, \max(F) +1] \smallsetminus F.
\end{align*}\qed
\end{lemma}

We now consider the transformation $\sigma_i\colon \sPf(\NN) \to  \sPf(\NN)$ which adjoins to a set $F\in \sPf(\NN)$ the $i$th smallest element of its complement $\NN \smallsetminus F$. Informally: $\sigma_i$ fills the $i$th hole. Thus for instance $\sigma_1(F) = F \cup \{\min(\NN \smallsetminus F)\}$.

We have already defined the deletion operations
\begin{equation*}
\alpha_i \colon \sP_{k}(\NN) \longrightarrow \sP_{k-1}(\NN).
\end{equation*}
Note that $\alpha_i$ gives a partially defined transformation of $\sPt(\NN)$. More precisely, writing $\sP_{\geq n}(\NN)= \bigcup_{k\geq n} \sP_{k}(\NN)$, we can consider
\begin{equation}\label{eq_alpha_i}
\alpha_i \colon \mleft( \sPt(\NN) \cap \sP_{\geq i+1}(\NN) \mright) \longrightarrow \sPt(\NN).
\end{equation}
Indeed, in applying $\alpha_i$ to sets of at least $i+1$ elements, we preserve the pin-headed condition. Moreover, this behaves well with respect to the complementation map considered in \Cref{lemma:delete-fill}. Specifically, one checks readily the following, which makes precise the heuristic that taking complements replaces deleting entries by filling holes.

\begin{lemma}\label{lem:alpha-sigma}
On the set $\sPt(\NN) \cap \sP_{\geq i+1}(\NN)$, the transformation $\alpha_i$ is intertwined with $\sigma_i$ under the bijection $\sPt(\NN) \to \sPf(\NN)$ of \Cref{lemma:delete-fill}.\qed
\end{lemma}

As before, we continue to identify finite increasing sequences with finite sets for the following statement and its proof.

\begin{proposition}
For all $k\geq 1$, the mean $\nu_k$ constructed in \Cref{sec:idem} satisfies $\nu_k(\sPt(\NN))=1$. In particular, $\nu_k$ can be viewed as a mean on $\sP_k(\NN) \cap \sPt(\NN)$.
\end{proposition}

\begin{proof}
Let $f_k$ be the characteristic function of the set $\sP_k(\NN) \smallsetminus \sPt(\NN)$; the statement amounts to proving $\nu_k(f_k)=0$. For $k=1$, we have $f_1 = 0$. For $k\geq 2$, we see that $f_k^x$ vanishes unless $x=1$. Therefore
\begin{equation*}
\nu_k(f_k) =  \nu \mleft( x \mapsto \nu_{k-1}\mleft( f_k^x\mright) \mright)  = 0
\end{equation*}
because $\nu(\{1\})=0$ ($\nu$ is diffuse).
\end{proof}

We can now apply \Cref{thm:omega-invariant} to deduce the following corollary. Once again, we clarify this in the spirit of \Cref{rem:subset}: a priori $\alpha_i \nu_{k+1}$ is a mean on $\sP_{k+1}(\NN)$ (since we only know that $\alpha_i(\sP_{k+1}(\NN) \cap \sPt(\NN))$ lies in $\sPt(\NN)$ if $i\leq k$), so we show the identity $\alpha_i \nu_{k+1} = \nu_{k}$ as means on $\sP_k(\NN)$.

\begin{corollary}\label{cor:alpha-invariant}
The sequence of means $\nu_k$ on $\sP_k(\NN) \cap \sPt(\NN)$ satisfies $\alpha_i \nu_{k+1} = \nu_{k}$ for all $k\geq 1$ and all $1\leq i \leq k+1$.\qed
\end{corollary}

Our second reversal now gives the corresponding statement for the everywhere-defined maps $\sigma_i$.

\begin{corollary}\label{cor:sigma-equivariant}
There is a sequence of means $\mu_k$ on $\sPf(\NN)$ satisfying $\sigma_i \mu_{k+1} = \mu_{k}$ for all $k\geq 1$ and all $1\leq i \leq k$.
\end{corollary}

\begin{proof}
Define $\mu_k$ to be the image of the mean $\nu_k$ of \Cref{cor:alpha-invariant} under the bijection $\sPt(\NN) \to \sPf(\NN)$ of \Cref{lemma:delete-fill}. Then \Cref{lem:alpha-sigma} shows that  $\sigma_i \mu_{k+1} = \mu_{k}$ holds indeed.
\end{proof}

We can now record the main result of this section.

\begin{corollary}\label{cor:sigma-invariant}
There is a mean $\mu$ on $\sPf(\NN)$ satisfying $\sigma_i \mu = \mu$ for all $i\geq 1$.
\end{corollary}

\begin{proof}
Consider the sequence $\bar \mu_n$ of means defined by the Ces{\`a}ro averages $\bar \mu_n = \frac1n \sum_{k=1}^n \mu_k$, where $\mu_k$ is given by \Cref{cor:sigma-equivariant}. Then any accumulation point $\mu$ of this sequence in the compact space of means has the required property.
\end{proof}

In the diffuse spirit of \Cref{sec:idem}, we could have replaced the above use of Ces{\`a}ro averages and accumulation points by one more Arens product, namely the mean-integration of the various $\bar \mu_n$ over $n$ for some invariant mean on $\NN$.

\section{The facial monoid}\label{sec:facial}
Consider the monoid presentation
\begin{equation*}
S = \big\langle s_i : s_j s_i = s_i s_{j+1} \ \forall\, 1\leq i \leq j  \big\rangle.
\end{equation*}
Thus $S$ is the monoid of all possible words (=finite sequences) on the letters $s_i$ (with $i\in\NN$), quotiented by the equivalence relation generated by the relations $s_j s_i = s_i s_{j+1}$ applied to any substring of such words. There would be no substantial difference whatsoever working with semigroups rather than monoids; the only difference is that we allow the empty word $e$, which provides an identity element for $S$.

In order to manipulate elements of $S$ it is crucial that this presentation admits normal forms. We shall say that a word is a \textbf{flat normal form} if it contains no substring $s_j s_i $ with $j\geq i$. The terminology reflects the fact that in particular no proper power $s_i^{p>1}$ can occur; thus, any flat normal form is either empty, or a single letter, or of the form $s_{i_1} \cdots s_{i_n}$ with $i_1 < \cdots < i_n$. In conclusion, flat normal forms can be identified with $\sPf(\NN)$.

The following is an elementary instance of the theory of string rewriting systems.

\begin{lemma}\label{lem:flat-normal}
Every element of the monoid $S$ admits a unique representative in flat normal form.
\end{lemma}

\begin{proof}
We consider the string rewriting system defined by the following rules:
\begin{equation*}
s_j s_i \to s_i s_{j+1} \kern5mm \forall\, 1\leq i \leq j.
\end{equation*}
The statement is now a very simple case of Newman's Lemma~\cite{Newman} (for which a more streamlined reference is~\cite[Lem.~2.4]{Huet}). Specifically, according to Newman's Lemma it suffices to show that the system is both terminating and locally confluent. 

\textbf{Local confluence}, also called the weak Church--Rosser property, means that for all $w, w', w''$ with $w\to w'$ and $w \to w''$ there exists $v$ with $w' \to^* v$ and $w'' \to^* v$. Here the symbol $\to$ denotes an application of a rewriting rule to any substring of the given string and $\to^*$ is the reflexive transitive closure of $\to$. Moreover, it suffices to check the case where $w\to w'$ and $w \to w''$ are applications of rewriting rules to substrings of $w$ that overlap. The only such overlaps are of the form $s_k s_j s_i$ when $i\leq j \leq k$. We then check
\begin{equation*}
\xymatrix @R-2pc { & s_k s_i s_{j+1} \ar[r] &  s_i  s_{k+1}s_{j+1} \ar[dr] & \\
s_k s_j s_i \ar[ur] \ar[dr]&&&  s_i  s_{j+1}s_{k+2}\\
 & s_j s_{k+1} s_i  \ar[r] &  s_j  s_i s_{k+2}\ar[ur] &}
\end{equation*}
as desired.

As to \textbf{terminating}, this means that there is no (finite) string admitting an infinite sequence of successive applications of the rules to its substrings. To verify this, we associate to any string $s_{i_1} \cdots s_{i_n}$ (of length $n\geq 2$) the sequence of $n-1$ integers $h_1, \ldots, h_{n-1} \geq 0$ defined by $h_j = |\{ k : k>j \text{ and } i_k \leq i_j\}|$. Then any application of the rules will decrease the lexicographic order on the corresponding sequences of integers. This implies termination since the lexicographic order is well-founded.
\end{proof}

\Cref{lem:flat-normal} allows us to identify $S$ with $\sPf(\NN)$; next, we need to keep track of the monoid structure in the model $\sPf(\NN)$.

\begin{lemma}\label{lem:S-fin}
Left multiplication by $s_i$ corresponds to the hole-filling operation $\sigma_i$ of \Cref{sec:up-down}.
\end{lemma}

\begin{proof}
On flat normal forms, this is a direct verification.
\end{proof}

We record another consequence of the flat normal forms.

\begin{lemma}\label{lem:right-cancel}
$S$ is right cancelable; that is, for every $s, t, x\in S$ we have $sx=tx \Rightarrow s=t$.
\end{lemma}

\begin{proof}
It suffices to show this for $x=s_i$ for all $i$; it is then a direct observation using the flat normal forms of $s, t$.
\end{proof}

Consider the sub-monoid $S_n$ of $S$ generated by the first $n$ generators $s_1, \ldots , s_n$. Of course the relations $s_j s_i = s_i s_{j+1}$ still hold in $S_n$ when $1\leq i \leq j \leq n-1$; the content of \Cref{lem:monoid-inj} below is that $S_n$ is not subject to any further ``hidden'' relations imposed by $S$. The reason this is not immediately obvious is that the flat normal forms are not well suited to $S_n$. For instance, we have
\begin{equation*}
s_j s_n s_i = s_i s_n s_{j+1}
\end{equation*}
but these two words are not confluent when the rewriting rules used to prove \Cref{lem:flat-normal} are restricted to $S_n$.

\begin{lemma}\label{lem:monoid-inj}
The canonical homomorphism
\begin{equation*}
\big\langle s_i, 1\leq i \leq n : s_j s_i = s_i s_{j+1} \ \forall\,  i \leq j \leq n-1 \big\rangle \longrightarrow S_n
\end{equation*}
is injective.
\end{lemma}

\begin{proof}

We denote by $\widetilde S_n$ the monoid given by the presentation on the left hand side of the statement. We consider the rewriting system for $S$ with the opposite of the rules used before, as follows:
\begin{equation*}
s_i s_j \to s_{j-1} s_i \kern5mm \forall\, 1\leq i < j.
\end{equation*}
It is terminating because these rules decrease the sum of all indices appearing in a string. The system is locally confluent by virtually the same verification as for flat normal forms. Namely, the only case to check arises for substrings of the form $s_i s_j s_k$ with $1\leq i < j < k$:
\begin{equation*}
\xymatrix @R-2pc { & s_i s_{k-1} s_j \ar[r] &  s_{k-2}s_i  s_j \ar[dr] & \\
s_i s_j s_k \ar[ur] \ar[dr]&&&  s_{k-2} s_{j-1} s_i\\
 & s_{j-1} s_i s_k \ar[r] &  s_{j-1}s_{k-1} s_i \ar[ur] &}
\end{equation*}
Thus there is a unique normal form, which we call the \textbf{descending normal form} because it consists of all words $s_{i_1} \cdots s_{i_n}$ with $i_1 \geq \cdots \geq i_n$. In contrast to the flat normal form, powers do occur, i.e. the sequence of indices is not always strictly decreasing.

The most important difference, however, is that the proof of confluence holds unchanged if we restrict the generators and the relations to $i,j\leq n$, because the indices never increase under the application of these reversed rules.

Therefore, $\widetilde S_n$ also admits unique descending normal forms for every element. This realises $\widetilde S_n$ as a subset of $S$, namely $S_n$.
\end{proof}

\begin{proposition}\label{prop:Sn}
For every $n\geq 2$, the sub-monoid $S_n < S$ is non-amenable.
\end{proposition}

\begin{proof}
As explained in the Introduction, the monoid given by the presentation
\begin{equation*}
\big\langle s_i, 1\leq i \leq n : s_j s_i = s_i s_{j+1} \ \forall\,  i \leq j \leq n-1 \big\rangle
\end{equation*}
admits fixed-point free actions on convex compact spaces, i.e. it is non-amenable. This monoid coincides with $S_n$ by \Cref{lem:monoid-inj}.
\end{proof}

\section{Amenability and random finite sets}\label{sec:amen}
Let $S$ be any monoid, or more generally any semi-group. A \textbf{convex compact $S$-space} refers to a convex compact set $K$ in some ambient Hausdorff topological vector space $V$, together with an $S$-action on $V$ by continuous affine transformations preserving $K$.

In most of the literature, $V$ is supposed locally convex. In that case it is equivalent, and perhaps more natural, to only require the $S$-action to be defined on $K$. Indeed, $K$ then is naturally embedded in the dual of the Banach space of continuous affine maps $K\to \RR$, endowed with the weak-* topology. This dual then provides a natural replacement for $V$ which is endowed with a linear representation of $S$ preserving $K$.

\begin{theorem}[Day~\cites{Day57, Day61}]\label{thm:Day}
The following are equivalent.
\begin{enumerate}
\item Every non-empty convex compact $S$-space has a fixed point.\label{pt:Day:fixed}
\item There exists a left-invariant mean on $S$.\label{pt:Day:mean}
\item For every finite set $A\subseteq S$ and every $\epsilon >0$ there is a finitely supported probability distribution $\mu$ on $S$ such that $\| s \mu -  \mu \|_1 < \epsilon$ holds for all $s\in A$.\label{pt:Day:Reiter}
\end{enumerate}
\end{theorem}

When these equivalent conditions are satisfied, we say that $S$ is \textbf{amenable}.

We caution the reader that Day called this property \emph{left amenability}, reserving the term amenability for the case where $S$ admits a mean invariant under both left and right translations. Furthermore, he proved that the latter condition is equivalent to admitting both a left invariant mean and a right invariant one~\cite[\S6]{Day57}. We shall simply refer to his ``right amenability'' as the amenability of the opposite semi-group $S^\mathrm{op}$.

\begin{remark}\label{rem:Sop}
In the case of the facial monoid $S$ that we study, the amenability of  $S^\mathrm{op}$ is trivial: any point $p$ fixed by $s_1\in S^\mathrm{op}$ is fixed by $S^\mathrm{op}$. Indeed, using the right action notation we have
\[
p s_i = (p s_1^{i-1}) s_i = (p s_1 s_2 \cdots s_{i-1}) s_i = p s_1^i = p.
\]
Thus \Cref{thm:fixed-pt} establishes at once amenability in every sense of the word.

Notice furthermore that the above computation holds for any $S^\mathrm{op}$-action on any sort of set. This implies, for instance, that $S^\mathrm{op}$ satisfies the Schauder--Tychonoff fixed-point principle.
\end{remark}

\begin{proof}[About the proof of \Cref{thm:Day}]
The implication \eqref{pt:Day:fixed}$\Longrightarrow$\eqref{pt:Day:mean} is obvious. The implication \eqref{pt:Day:mean}$\Longrightarrow$\eqref{pt:Day:Reiter} requires some work; it is proved by Day in Theorem~1 of~\cite[\S5]{Day57}, though he treats simultaneously left and right invariance. Note that this implication was already discussed in the proof of \Cref{thm:random-sets}. As for \eqref{pt:Day:Reiter}$\Longrightarrow$\eqref{pt:Day:fixed}, it really is an expansion of Kakutani's proof~\cite{Kakutani38}. Namely, given $x\in K$, consider  $\mu x = \sum_{s\in S} \mu(s) s x$, which gives a net in $K$ as $A$ increases and $\epsilon$ decreases. Then any accumulation point of this net is fixed. Although this is usually stated in the locally convex setting, it is verified directly in any (Hausdorff) topological vector space. This is actually how it is done for a single transformation in~\cite[V.10.6]{Dunford-Schwartz_I} and the general case of $\mu$ is identical.
\end{proof}

\begin{proof}[Proof of \Cref{thm:fixed-pt}]
We have an identification of $S$ with $\sPf(\NN)$ as described by \Cref{lem:flat-normal} and \Cref{lem:S-fin}. Under this identification, \Cref{cor:sigma-invariant} states precisely that $S$ admits a left-invariant mean.
\end{proof}

It is well-known that F\o lner conditions are more delicate for monoids than for groups. On the one hand, the amenability of $S$ established in \Cref{thm:fixed-pt} implies that for any $\epsilon>0$ and any finite subset $S_0\subseteq S$ there is a finite set $A\subseteq S$ that is ``almost invariant'' in the sense that
\begin{equation*}
|s A \smallsetminus A| < \epsilon |A| \kern3mm \text{for all }s\in S_0.     
\end{equation*}
This was established by Frey~\cite[\S6]{FreyPhD}. On the other hand, that condition is not sufficient for amenability. Perhaps more surprisingly, that condition is not difference-symmetric, and indeed for $S$ one can prove that it is impossible to find finite sets $A\subseteq S$ for which the symmetric difference $s A \triangle A$ is arbitrarily small in relation to $A$. More precisely:

\begin{proposition}
For every finite set $A\subseteq S$ we have
\begin{equation*}
\text{either}\kern3mm |A \smallsetminus s_1 A|\geq \tfrac15 |A| \kern3mm \text{or} \kern3mm |A \smallsetminus s_2 A|\geq \tfrac15 |A|.
\end{equation*}
\end{proposition}

\begin{proof}
Since $S$ is right cancelable by \Cref{lem:right-cancel}, we can apply the argument given in~\cite[Thm.~2.2]{Klawe77}. More precisely, on page 94 thereof, choose $r=s=s_1$, $t=s_2$, and observe that $S'=S$ by right cancelability.
\end{proof}

The straightforward generalization of semi-direct products of groups to monoids is as follows. Let $S,T$ be monoids, where $S$ moreover acts on $T$ by endomorphisms. Then $T\rtimes S$ is the direct product of endowed with the associative multiplication
\begin{equation*}
(t, s) (t', s') = \big(t s(t'), s s'\big).
\end{equation*}
By construction, both $S$ and $T$ are embedded in $T\rtimes S$ and there is a retraction quotient morphism to $S$.

If $S$ and $T$ are amenable, it does not necessarily follow that $T\rtimes S$ is so. Given a compact convex $(T\rtimes S)$-set $K$, we see that $s\in S$ will map the $T$-fixed points $K^T$ to $K^{s(T)}$. This shows amenability e.g.\ when $s(T)=T$, that is, when $S$ acts by surjective endomorphisms. (Another proof of that fact can be found in~\cite{Klawe77}.)

We shall consider another semi-direct product, which we call the \textbf{right semi-direct} product $S \ltimes T$. The setting is now that $S$ has a \emph{right} action on $T$ (still by endomorphisms), which we denote by $t^s$. (In other words, $S^\mathrm{op}$ acts on $T$.) The direct product set has an associative multiplication
\begin{equation*}
(s, t) (s', t') = \big(s s',  t^{s'} \, t'\big).
\end{equation*}
Again, $S$ and $T$ are embedded in $S \ltimes T$ and there is a projection morphism to $S$.

This time, $S$ does preserve $K^T$ and therefore, without any further assumption, $S \ltimes T$ is amenable when $S$ and $T$ are so. (The bridge to the results of~\cite{Klawe77} is to write $S \ltimes T=(T^\mathrm{op} \rtimes S^\mathrm{op})^\mathrm{op}$.)

\medskip

We now specialize to the case where $S$ is the facial monoid and $T$ is the free abelian monoid on the set $\NN$ (or, alternatively, the free abelian group $\ZZ[\NN]$ on $\NN$). We write $t\in T$ as $t=(t_n)_{n\in\NN}$, where all but finitely many of the integers $t_n$ are zero. We consider the right $S$-action defined by
\begin{equation*}
  (t^{s_i})_n = 
    \begin{cases}
        t_n &\text{ if } n\leq i, \\
        0 &\text{ if } n = i+1, \\
         t_{n-1} &\text{ if } n\geq i+2.
    \end{cases}
\end{equation*}

Since $S$ is amenable by \Cref{thm:fixed-pt} and $T$ is so by the Markov--Kakutani theorem, we deduce:

\begin{corollary}
The right semi-direct product $S\ltimes T$ is amenable.\qed
\end{corollary}

Our interest in that statement stems from the following fact.

\begin{proposition}
There is a monoid embedding of the Thompson monoid $F_+$ into $S\ltimes T$.

Moreover, that embedding is a lift for the canonical quotient $F_+ \to S$.
\end{proposition}

\begin{proof}
Recall that every element $g\in F_+$ can be uniquely written as a (possibly empty) sequence
\begin{equation*}
g = g_1^{p_1} g_2^{p_2} \cdots g_d^{p_d}
\end{equation*}
with $p_i$ non-negative integers, $p_d\neq 0$. This can be proved similarly to \Cref{lem:flat-normal}; another proof, for $F$, is given in~\cite[2.7]{Cannon-Floyd-Parry}. The embedding is now given by
\begin{equation*}
g\mapsto \big( \bar g, (p_n)_{n\in\NN} \big)
\end{equation*}
wherein $\bar g$ is the image of $g$ under canonical quotient $F_+ \to S$. The fact that this is indeed a homomorphism is a direct verification by induction on $d$; it is injective because $p$ records the normal form of $g$.
\end{proof}

\section{Ergodic theory of the face relations:\texorpdfstring{\\}{} a De Finetti type Theorem}\label{sec:Finetti}

Theorem~\ref{thm:random-sets} shows that there are probability measures on the set of finite subsets of the natural numbers that are almost-invariant to deleting the $i$\textsuperscript{th} smallest element. More generally, one could consider the set of nonempty, strictly increasing sequences of integers, with the operation of deleting the $i$\textsuperscript{th} element of the sequence. Clearly, there cannot be a completely invariant probability measure on this space. However, if one considered instead the space $\NN^\NN$ of \emph{all} sequences of integers, there is a probability measure with that property. Indeed, one could take the point mass on a constant sequence. Beyond that, any product measure is invariant to deletions. This section is motivated by the question of what other deletion-invariant measures exist.

\medskip
We consider more generally a measurable space $(X,\Sigma)$, and let $\Omega = X^\NN$ be the set of sequences taking value in $X$. For $i \in \NN$, let $s_i \colon \Omega \to \Omega$ be given by
\begin{align}
    \label{eq:deletion}
    [s_i \omega]_n = 
    \begin{cases}
        \omega_n&\text{if } n<i \\
        \omega_{n+1}&\text{otherwise.}
    \end{cases}
\end{align}
Note that these transformation satisfy the face relations. The transformation $s_1$ is the usual one-sided shift: it deletes the first coordinate and shifts the rest to the left. The transformation $s_i$ leaves coordinates $1$ through $i-1$ unchanged, deletes the $i$\textsuperscript{th} coordinate, and shifts the rest to the left. 

Let $\mu$ be a probability measure on $\Omega$. The following definitions are standard: We say that $\mu$ is \textbf{invariant} if for every measurable $E \subseteq \Omega$ and every $s_i$ it holds that $\mu(s_i^{-1} E) = \mu(E)$. We say that $\mu$ is an  \textbf{IID} measure if there exists a probability measure $\eta$ on  $(X,\Sigma)$ such that $\mu = \eta^\NN$; clearly, IID measures are invariant. We say that $\mu$ is \textbf{ergodic} if for a measurable $E \subseteq \Omega$ such that $s_i^{-1}E=E$ for all $i$ it holds that $\mu(E) \in \{0,1\}$.
\begin{proposition}
    \label{prop:ergodic}
    Every invariant ergodic probability measure on $X$ is IID.
\end{proposition}

\begin{proof}
    Let $\mu$ be invariant and ergodic. Let $\pi_1 \colon \Omega \to X$ be the projection $\pi_1(\omega)=\omega_1$, and let $\eta = \pi_1\mu$ be the law of $\omega_1$. That is, for measurable $Y \subseteq X$, $\eta(Y) = \mu(\{\omega\,:\, \omega_1 \in Y\})$. 
    
    We will show that if $Y_1,Y_2,\ldots$ are subsets of $X$ then 
    \begin{align}
        \label{eq:product}
        \mu\left(\prod_i Y_i\right) = \prod_i \eta(Y_i).
    \end{align}

    Fix any measurable $Y \subseteq X$. For $i \in \NN$, let
    \begin{align*}
      F_i = \{\omega \,:\, \omega_i \in Y\}.
    \end{align*}
    Since $F_i = s_1^{1-i}F_1$ and since $\mu$ is $s_1$-invariant, $\mu(F_i) = \mu(F_1) = \eta(Y)$.

    Fix any measurable $W \subseteq X^{M-1}$, and let
    \begin{align*}
      G = \{\omega\,:\, (\omega_1,\ldots,\omega_{M-1}) \in W\}.
    \end{align*}
    We claim that $\mu(F_M \cap G) = \mu(F_M) \mu(G)$. This will then imply~\eqref{eq:product} by a standard argument. Indeed, by the sigma-additivity of probability measures it suffices to show that if $Y_1,\ldots,Y_M$ are subsets of $X$, and if $A_i = \{\omega\,:\,\omega_i \in Y_i\}$, then $\mu(A_i\cap \cdots \cap A_M) = \eta(Y_1)\cdots\eta(Y_M)$. To this end, it suffices to show that $\mu(A_i\cap \cdots \cap A_M) = \mu(A_i\cap \cdots \cap A_{M-1})\eta(Y_M)$, as the proof then follows by induction.

Turning to the claim, let $f_i \in L^2(\Omega,\mu)$ be given by
    \begin{align*}
      f_i(\omega) = 1_{F_i}(\omega)-\mu(F_i) = 1_{F_i}(\omega)-\mu(F_1).
    \end{align*}
    Note that $\int f_i\,\dd\mu = 0$ and $\| f_i \|^2 =  \mu(F_1)(1-\mu(F_1)) \leq 1$. For $i\in \NN$, let $U_i$ be the isometric operator on $L^2(\Omega,\mu)$ given by $U_ig = g \circ s_i$. It follows that
    \begin{align*}
      U_i f_n =
      \begin{cases}
        f_n&\text{if } n < i\\
        f_{n+1}&\text{otherwise.}
      \end{cases}
    \end{align*}

    By the von Neumann ergodic theorem the limit
    \begin{align*}
      f = \lim_N \frac{1}{N}\sum_{k=0}^{N-1} U_M^k f_M
    \end{align*}
    exists and is equal to the projection of $f_M$ on the space of $U_M$-invariant elements of $L^2(\Omega,\mu)$. Since $U_M^k f_M = f_{M+k}$, we have that
    \begin{align*}
      f = \lim_N \frac{1}{N}\sum_{k=0}^{N-1} f_{M+k} = \lim_N \frac{1}{N}\sum_{k=0}^{N-1} f_{k},
    \end{align*}
    where the second equality follows from the fact that
    \begin{align*}
      \left\| \frac{1}{N}\sum_{k=0}^{N-1} f_{M+k} - \frac{1}{N}\sum_{k=0}^{N-1} f_{k} \right\| \leq \frac{2M}{N}.
    \end{align*}
    By the same argument, it holds for any $i \in \NN$ that
    \begin{align*}
      f = \lim_N \frac{1}{N}\sum_{k=0}^{N-1} f_{i+k} = \lim_N \frac{1}{N}\sum_{k=0}^{N-1} U_i^kf_{i+k}.
    \end{align*}
    Thus $f$ is $U_i$-invariant for all $i$, and it follows from ergodicity that $f$ is constant. Moreover, $\int f_{i}\,\dd\mu = 0$, and so $\int f\,\dd\mu = 0$. Thus $f=0$.

    Let $g(\omega) = 1_G(\omega)-\mu(G)$, and note that $U_Mg=g$, by the definition of $G$. Now,
    \begin{multline*}
      (g,f_M) = \int (1_G(\omega)-\mu(G))(1_{F_M}(\omega) - \mu(F_M))\,\dd\mu(\omega)\\
      = \mu(G \cap F_M) - \mu(G)\mu(F_M).
    \end{multline*}
    Since $U_M g = g$ and $U_M$ preserves inner products, it holds that $(g,U_M^kf_M) = (g,f_M)$. Thus also $(g,f)=\mu(G \cap F_M) - \mu(G)\mu(F_M)$. Since $f=0$, it follows that $\mu(G \cap F_M) = \mu(G)\mu(F_M)$ as claimed.
\end{proof}

Proposition~\ref{prop:ergodic} implies that if $\mu$ is ergodic and invariant, then the transformation $s_1$ is already itself ergodic. As we next show, this is not a coincidence. Suppose $(\Theta,\Sigma)$ is a measurable space, and let $s_1,s_2,\ldots$ be measurable transformations satisfying the face relations. Suppose that $f \colon \Theta \to \{0,1\}$ is the indicator of $E \subseteq \Theta$. Then the indicator of $s_i^{-1}E$ is $f \circ s_i$ and the action $(s_i,E) \mapsto s_i^{-1}E$ is in fact an action of $S^\mathrm{op}$. It follows that if $s_1^{-1}E=E$ then $s_i^{-1}E=E$ for all $i$, see \Cref{rem:Sop}. In other words:

\begin{claim}
  Let $(\Theta,\Sigma,\mu)$ be a probability space, and let $s_1,s_2,\ldots$ be measure-preserving transformation satisfying the face relations. Suppose that $\mu$ is $(s_1,s_2,\ldots)$-ergodic. Then $\mu$ is $s_1$-ergodic. \qed
\end{claim}

Suppose, as in the claim above, that $s_1, {s_2,\ldots}$ are measure-preserving transformations of $(\Theta,\Sigma,\mu)$ satisfying the face relations. One could ask if these relations imply some dynamical properties of this system, say assuming ergodicity. Of course, a particular case is $s_1=s_2=\cdots$, in which case the dynamical system can be any single measure-preserving transformation and no restrictions are possible. To exclude this possibility, assume that $s_2$ is sufficiently different from $s_1$: specifically, assume that $\mu$ is $s_1$-ergodic, but not $s_2$-ergodic. Under this assumption, we show that the system admits a non-trivial equivariant factor to $\{0,1\}^\NN$. Together with Proposition~\ref{prop:ergodic}, this implies that the system $(\Theta,\Sigma,\mu,s_1)$ has positive entropy.

\begin{proposition}
 Let $(\Theta,\Sigma,\mu)$ be a probability space, and let $s_1,s_2,\ldots$ be measure preserving transformation satisfying the face relations. Suppose that $\mu$ is $s_1$-ergodic but not $s_2$-ergodic. Then there is a map $\varphi \colon \Theta \to \{0,1\}^\NN$ such that $\varphi(s_i\theta) = s_i\varphi(\theta)$, with the r.h.s.\ given by \eqref{eq:deletion}, and such that $\varphi_*\mu$ is not a point mass.
\end{proposition}
\begin{proof}
    Fix some $s_2$-invariant $E \subseteq \Theta$ with $\mu(E) \in (0,1)$, let $f \colon \Theta \to \{0,1\}$ be the indicator of $E$, and let $\varphi \colon \Theta \to \{0,1\}^\NN$ be given by
\begin{align*}
    [\varphi(\theta)]_n = f(s_1^n \theta) = f \circ s_1^n (\theta).
\end{align*}

Since $f \circ s_2 = f$, it holds for all $i \geq 2$ that, likewise, $f \circ s_i = f$. Hence 
\begin{align*}
    f \circ s_1^n \circ s_i = 
    \begin{cases}
    f \circ s_1^n&\text{if } n < i\\
    f \circ s_1^{n+1}&\text{otherwise.}
    \end{cases}
\end{align*}
Thus
\begin{align*}
    [\varphi (s_i \theta)]_n = f \circ s_1^n \circ s_i(\theta) = 
    \begin{cases}
        [\varphi (\theta)]_n &\text{if } n<i\\
        [\varphi (\theta)]_{n+1}&\text{otherwise.}
    \end{cases}
\end{align*}
Hence $\varphi$ is an equivariant factor to $\{0,1\}^\NN$, with respect to the action given by \eqref{eq:deletion}. By Proposition~\ref{prop:ergodic} the push-forward measure $\varphi_*\mu$ is the product measure $\eta^\NN$, where $\eta = (1-\mu(E))\delta_0+\mu(E)\delta_1$. Finally, since $\mu(E) \in (0,1)$, it follows that $\varphi_*\mu$ is not a point mass.
\end{proof}

\appendix

\section{On explicit constructions for Theorem~\ref{thm:random-sets}}\label{sec:explicit}
The proof of Theorem~\ref{thm:omega-invariant} suggests a natural construction yielding an explicit version of Theorem~\ref{thm:random-sets}: apply the same construction, but use countably additive probability measures rather than means. The next proposition shows that this approach cannot work.

Consider the case that $\nu$ is a (countably additive) probability measure on $\NN$, and let $\nu_k$ be the probability measure on $\incr^k$ defined as in the preamble to Theorem~\ref{thm:omega-invariant}. Then $\nu_k$ is again countably additive, and is indeed the law of $k$ steps of the $\nu$-random walk, compare \Cref{rem:RW}.

\begin{proposition}
    \label{prop:no_random_walk}
    Suppose $\nu$ is a countably additive probability measure, and that $\alpha_{1}\nu_{k+1}$ is $\epsilon$-close to $\nu_k$ in total variation for some $k\geq 2$. Then $\epsilon > 1/4$.
\end{proposition}
\begin{proof}
  Let 
  \begin{align*}
      M = \min \{y \in \NN\,:\, \nu(\{1,\ldots,y\}) > 1/2\}.
  \end{align*}
  This is well-defined since $\nu$ is countably additive. Let $A \subseteq \incr^k$ be the event
  \begin{align*}
      A = \{(y_1,\ldots,y_k) \,:\, y_1 \leq M\},
  \end{align*}
  and let $B \subseteq \incr^{k+1}$ be the event
  \begin{align*}
      B = \{(y_1,\ldots,y_{k+1}) \,:\, y_2 \leq M\}.
  \end{align*}
  Then $B = (\alpha_{1})^{-1}A$, and so $\epsilon \geq |\nu_{k+1}(B)-\nu_k(A)|$. 
  
  By the definition of $M$, $\nu_k(A) = \nu(\{1,\ldots,M\}) > 1/2$. As for $B$, it is contained in the intersection of the two independent events
  \begin{align*}
      B_1 &= \{(y_1,\ldots,y_{k+1}) \,:\, y_1 \leq M-1\}\\
      B_2 &= \{(y_1,\ldots,y_{k+1}) \,:\, y_2-y_1 \leq M-1\}.
  \end{align*}
  In analogy to \Cref{rem:RW}, the definition of $M$ thus implies 
  \begin{align*}
      \nu_{k+1}(B_1) = \nu_{k+1}(B_2) = \nu(\{1,\ldots,M-1\}) \leq 1/2.
  \end{align*}
  Therefore it follows $\nu_{k+1}(B) \leq \nu(\{1,\ldots,M-1\})^2 \leq 1/4$ and we conclude $\epsilon > 1/4$.
\end{proof}

We note that this proof shows more generally that $\epsilon$ admits a positive lower bound as soon as the mean $\nu$ is not diffuse. 

In probabilistic terms, this proof shows that if $x_1,x_2,\ldots,x_k$ are independent and identically distributed random variables taking values in $\NN$, then the joint distribution of their partial sums $y_i = x_1+\cdots+x_i$ cannot be almost invariant to deletions. 

In fact, the same proof implies a more general result: if $x_1,x_2,\ldots,x_k$ are independent (but not necessarily identically distributed) random variables taking values in $\NN$, and if $*$ is a binary operation on $\NN$ satisfying (i) $x*y > x$ and (ii) $x\geq x'$, $y > y'$ implies $x * y > x' * y'$, then the joint distribution of the partial products $x_1 * \cdots * x_i$ cannot be almost invariant to deletions. This excludes a large range of Markov chains from constituting possible explicit constructions for Theorem~\ref{thm:random-sets}.

\medskip
\itshape
In the remainder of this section we outline an explicit construction that we conjecture  satisfies Theorem~\ref{thm:random-sets}.
\upshape

The measure $\mu$ will be given by $\mu = \frac{1}{n}\sum_{k=n+1}^{2n}\nu_k$, where each probability measure $\nu_k$ is finitely supported on subsets of $\NN$ of size $k$. We define $\nu_k$ as the image of the Lebesgue measure on the cube $[0,1]^k$ under the composition of three maps:
\begin{equation*}
[0,1]^k \xrightarrow{\ \Phi\ } \RR^k  \xrightarrow{\ E\ } \incrR^k \xrightarrow{\ \lfloor \cdot \rfloor\ } \incr^k.
\end{equation*}
As above,  $\incr^k$ is the set of strictly increasing sequences of length $k$ natural numbers, which we naturally identify with the finite subsets of $\NN$ of size $k$. We denote by  $\incrR^k$ the sequences of non-negative reals that increase by at least one.  We abusively suppress the index $k$ in our notation for $\Phi, E$ and $\lfloor \cdot \rfloor$ to emphasize that these maps are all truncations of the maps for $k=2n$, as follows. The first map is a product map
\begin{equation*}
 \Phi(x_1, \ldots, x_k) = \bigl(\exp\exp(r_1 x_1), \ldots, \exp\exp(r_k x_k)\bigr)
\end{equation*}
for some constants $r_1, \ldots, r_{2n} >1$.

The map $E$ is the exponential tower map given by
\begin{equation*}
    E(y_1,\ldots,y_k) = \left(y_1,y_1^{y_2}, y_1^{y_2^{y_3}}, \ldots, y_1^{y_2^{\iddots^{y_{k}}}}\right).
\end{equation*}
The map $\lfloor \cdot \rfloor$ is simply the component-wise floor map that rounds each coordinate down.

In probabilistic terms, we let $x_1,\ldots,x_{k}$ be IID random variables distributed uniformly on $[0,1]$, let $y_i = \exp\exp(r_ix_i)$, and let $z_1,\ldots,z_k$ be given by
\begin{align*}
    z_i = y_1^{y_2^{\iddots^{y_i}}}.
\end{align*}
Finally, $\nu_k$ is the distribution of $(\lfloor z_1 \rfloor,\cdots,\lfloor z_k \rfloor)$. Note that the rounding-down operation commutes with the deletion maps $\alpha_i$, since it is order preserving.

To choose the constants $r_i$ we start by choosing $r_{2n}$ to be large, and then, proceeding backwards, choose each $r_{i-1}$ to be much larger than $\exp(r_{i})$. Then $\exp(r_{i-1})$ is much larger than $\exp\exp(r_i)$, and so
\begin{align*}
    z_2 = y_1^{y_2} = \exp\left(\exp(r_1 x_1)+\exp\exp(r_2 x_2)\right)
\end{align*}
has approximately the same distribution as $z_1=y_1=\exp\exp(r_1 x_1)$. Indeed, the total variation distance between the distributions of $z_i$ and $z_j$ can be all made uniformly arbitrarily small by choosing $r_1,\ldots,r_{2n}$ appropriately.

The recursive structure of this construction means that proving that it is $\alpha_1$-almost invariant would imply that it is almost invariant to $\alpha_i$ for $i=1,\ldots,k$. That is, it suffices to show that the joint distribution of $(z_1,\ldots_,z_{k-1})$ is approximately equal to that of $(z_2,\ldots,z_k)$. By the comment above, the marginal distributions are indeed close. It is furthermore easy to show that the joint distribution of $(z_1,z_2)$ is close to that of $(z_2,z_3)$, and likewise that the joint distribution of $(z_1,z_2,z_3)$ is close to that of $(z_2,z_3,z_4)$. However, the general proof seems elusive.

\section{A motley mishmash of monotone monoids}\label{sec:related}
Both $S$ and $S^\mathrm{op}$ have natural representations as order-preserving maps of $\NN$. The purpose of this appendix is to place this into a wider context and to discuss the corresponding amenability questions.

\medskip

Let $\End_\leq(\NN)$ be the monoid of all order-preserving maps $\NN\to\NN$. We consider its sub-monoids $\Emb_\leq(\NN)$ and $\Epi_\leq(\NN)$ consisting repsectively of injective and surjective maps. We note that $\Emb_\leq(\NN)$ can also be realized as the ``monoid of subsequences'', where the product is given by iterating subsequences; it appeared in this form in~\cite{FFMN}. In that reference, it is proved that $\Emb_\leq(\NN)$ is boundedly acyclic; the \emph{non}-amenability statements in \Cref{thm:related} below sugest that there is no obvious shortcut to that acyclicity.

Both $\Emb_\leq(\NN)$ and $\Epi_\leq(\NN)$ are contained in the sub-monoid $\End_\leq^\infty(\NN)$ of unbounded maps (equivalently: maps tending to infinity). We are less interested in its complement, the countable set $\End_\leq^\mathrm{fin}(\NN)$ of maps with finite image, but we record that $\End_\leq^\mathrm{fin}(\NN)$ is also a sub-semigroup, indeed a two-sided ideal in $\End_\leq(\NN)$.

\medskip

Next we observe that $\End_\leq(\NN)$ is a Polish \emph{topological monoid} for the topology of pointwise convergence. All three of $\Emb_\leq(\NN)$, $\Epi_\leq(\NN)$ and $\End_\leq^\infty(\NN)$ are so too for the induced topology because they are $G_\delta$-sets. We recall that a topological monoid is \textbf{amenable} if it has the fixed-point property (as in \Cref{sec:amen}) but retricted to \emph{jointly continuous} actions on convex compact spaces.

\begin{theorem}\label{thm:related}
The four Polish monoids $\End_\leq(\NN)$, $\Emb_\leq(\NN)$, $\Epi_\leq(\NN)$ and $\End_\leq^\infty(\NN)$ are amenable as topological monoids, but all four are non-amenable as (abstract) monoids.
\end{theorem}

\begin{remark}
The main significance of this theorem regards $\Epi_\leq(\NN)$, where topological amenability will be deduced from \Cref{thm:fixed-pt} and non-amenability requires an argument.

The three other monoids are simpler: we shall see that a stronger form of topological amenability follows readily as in \Cref{rem:Sop} above, whereas non-amenability is immediately apparent from the action on $\NN$.
\end{remark}

In preparation for the proof we introduce one more sub-monoid of interest: the monoid $\End_\leq^\mathrm{tr}(\NN)$ of maps $f\colon\NN\to\NN$ that are eventually translations in the following sense:
\begin{equation}\label{eq:trans}
\exists N \in \NN \ \exists k\in \ZZ \ \forall n\geq N: f(n) = n+k.
\end{equation}
Furthermore, we have representations
\begin{equation*}
S, S^\mathrm{op} \longrightarrow \End_\leq^\mathrm{tr}(\NN) \ \subset \ \End_\leq^\infty(\NN) \ \subset \ \End_\leq(\NN)
\end{equation*}
defined as follows. Consider for $i\in \NN$ the map $s_i\colon \NN\to \NN$ defined by
\begin{equation*}
 s_i(n) = 
    \begin{cases}
        n &\text{ if } n \leq i, \\
        n-1 &\text{ if } n \geq i+1.
    \end{cases}
\end{equation*}
Then $s_i$ lies in $\Epi_\leq(\NN) \cap \End_\leq^\mathrm{tr}(\NN)$ and the facial relations hold. Thus this gives a representation
\begin{equation*}
S \longrightarrow \Epi_\leq^\mathrm{tr}(\NN) := \Epi_\leq(\NN) \cap \End_\leq^\mathrm{tr}(\NN);
\end{equation*}
one can see on the normal forms that this representation is an embedding of $S$.

On the other hand, we consider  $t_i\colon \NN\to \NN$ defined by
\begin{equation*}
 t_i(n) = 
    \begin{cases}
        n &\text{ if } n \leq i-1, \\
        n+1 &\text{ if } n \geq i.
    \end{cases}
\end{equation*}
This time we obtain similarly a representation
\begin{equation*}
S^\mathrm{op} \longrightarrow \Emb_\leq^\mathrm{tr}(\NN) := \Emb_\leq(\NN) \cap \End_\leq^\mathrm{tr}(\NN).
\end{equation*}

\begin{proof}[Proof of \Cref{thm:related}]
We begin with the (topological) amenability statements. We claim that the embedding of $S$ is \emph{onto} $\Epi_\leq^\mathrm{tr}(\NN)$. Indeed, given $f\in \Epi_\leq^\mathrm{tr}(\NN)$, consider the multiplicities $m_i=|f^{-1}(\{i\})|$. Then $m_i\geq 1$ with equality for all but finitely many $i\in\NN$. Now $f$ can be written as
\begin{equation*}
f =   s_N^{m_N -1} \,  s_{N-1}^{m_{N-1} -1}  \cdots  s_2 ^{m_2 -1}  \,  s_1^{m_1 -1} 
\end{equation*}
with any $N$ large enough; this witnesses that $f$ lies in the image of $S$. We recovered here the descending normal forms of \Cref{sec:facial}.

On the other hand, $\Epi_\leq^\mathrm{tr}(\NN)$ is dense in $\Epi_\leq(\NN)$. Indeed, given $N\in \NN$, any element $f\in \Epi_\leq(\NN)$ can be approximated by $f_N\in \Epi_\leq^\mathrm{tr}(\NN)$ defined by
\begin{equation*}
 f_N(n) = 
    \begin{cases}
        f(n) &\text{ if } n \leq N, \\
        f(N) + n-N &\text{ if } n \geq N+1.
    \end{cases}
\end{equation*}
In view of \Cref{thm:fixed-pt}, it follows that $\Epi_\leq(\NN)$ is amenable as a topological monoid. 

We turn to the other monoids. Observe first that for every $f\in \End_\leq^\mathrm{tr}(\NN)$ we have
\begin{equation*}
f \, t_1^{m} = t_1^{m+k}
\end{equation*}
whenever $m\geq N-1$, where $N$ and $k$ are as in~\eqref{eq:trans}. This is well defined: $m+k\geq 0$ because~\eqref{eq:trans} forces $N+k\geq 1$. As in \Cref{rem:Sop}, this shows that any $t_1$-fixed point in any action of $\End_\leq^\mathrm{tr}(\NN)$ or of $\Emb_\leq^\mathrm{tr}(\NN)$ will be fixed by that larger monoid. On the other hand, the approximation $f_N$ fiven by the exact same formula as above shows that $\Emb_\leq^\mathrm{tr}(\NN)$ is dense in $\Emb_\leq(\NN)$ and that $\End_\leq^\mathrm{tr}(\NN)$ is dense in $\End_\leq(\NN)$, a fortiori also in $\End_\leq^\infty(\NN)$. This shows that these three Polish monoids have the fixed point property even for possibly non-affine actions, applying Schauder--Tychonoff to $t_1$.

\medskip
We now turn to the non-amenability statements. The simplest cases are $\End_\leq(\NN)$, $\Emb_\leq(\NN)$, and $\End_\leq^\infty(\NN)$ because none of them admits an invariant mean on $\NN$ for the canonical action. Indeed they contain the elements $n\mapsto 2n$ and $n\mapsto 2n+1$, which have disjoint images.

It remains to show that $\Epi_\leq(\NN)$ is non-amenable. Consider the convex compact space $K$ of \emph{diffuse} means on $\NN$; then $\Epi_\leq(\NN)$ preserves $K$ since pre-images of finite sets under maps in $\Epi_\leq(\NN)$ remain finite. Suppose for a contradiction that $\mu\in K$ is fixed by $\Epi_\leq(\NN)$. Given $p\in \NN$, any congruence class modulo $p$ has $\mu$-mass $1/p$. Indeed, $s_1$ permutes these congruence classes up to finite subsets, and $\mu$ ignores finite subsets. Now consider the waltz map $w\in \Epi_\leq(\NN)$ defined by
\begin{equation*}
 w(n) = 
    \begin{cases}
        2 n/3 &\text{ if } n\equiv 0 \mod 3, \\
        2 \lfloor n/3 \rfloor +1  &\text{ otherwise},
    \end{cases}
\end{equation*}
wherein $\lfloor \cdot \rfloor$ denotes rounding down. The pre-image of the even numbers under $w$ consists exactly of the multiples of three, which contradicts the previous property regarding congruence classes.
\end{proof}

Finally, we record one more structure on the monoid $\End_\leq^\infty(\NN)$, namely the pseudo-inverse $f\mapsto f^-$ defined by
\begin{equation*}
f^-(x) = \min\{ y \in \NN : f(y) \geq x \}.
\end{equation*}
This pseudo-inverse will clarify the respective positions of $\Emb_\leq(\NN)$ and $\Epi_\leq(\NN)$ in $\End_\leq^\infty(\NN)$, as well as the representations of $S$ and $S^\mathrm{op}$.

We first list a few properties that are straightforward verifications.

\begin{proposition}\label{prop:inverse}
Let $f, g\in \End_\leq^\infty(\NN)$.
  \begin{enumerate}
\item $(f g)^- = g^- f^-$.
\item $f^- f \leq \mathrm{Id}$, $f f^- \geq \mathrm{Id}$.
\item $f^- f = \mathrm{Id} \ \Longleftrightarrow\ f\in \Emb_\leq(\NN)$.
\item $f f^- = \mathrm{Id} \ \Longleftrightarrow\ f\in \Epi_\leq(\NN)$.
  \end{enumerate}\qed
\end{proposition}

If follows that this anti-endomorphism switches $\Emb_\leq(\NN)$ and $\Epi_\leq(\NN)$.

\begin{corollary}
Let $f\in \End_\leq^\infty(\NN)$.
\begin{enumerate}
\item $ f\in \Emb_\leq(\NN)\ \Longrightarrow\ f^-\in \Epi_\leq(\NN)$.
\item $ f\in \Epi_\leq(\NN)\ \Longrightarrow\ f^-\in \Emb_\leq(\NN)$.
  \end{enumerate}
\end{corollary}

\begin{proof}
Suppose $f\in \Emb_\leq(\NN)$. We use \Cref{prop:inverse} as follows:
\begin{equation*}
\mathrm{Id} = \mathrm{Id}^- = (f^- f )^- = f^- (f^-)^-.
\end{equation*}
Thus $f^-$ admits a right inverse, which shows that it is onto. The proof of the second implication is the mirror image of this one.
\end{proof}

Returning to the embeddings of both $S$ and $S^\mathrm{op}$ that we exhibited, we check that the pseudo-inverse also intertwines them up to a shift in the indices, and we have:
\begin{equation}\label{eq:intertwine}
t_i^- = s_{i}, \ s_i^- = t_{i+1}; \ \ s_i t_i = s_i t_{i+1} = \mathrm{Id}.
\end{equation}
Putting everything together we deduce:

\begin{proposition}\label{prop:big-picture}
The representations of $S$ and $S^\mathrm{op}$ are isomorphisms onto $\Epi_\leq^\mathrm{tr}(\NN)$, respectively onto $\Emb_\leq^\mathrm{tr}(\NN)$, and the diagram of morphisms and anti-morphisms 
\begin{equation*}
\xymatrix{S \ar^-{\cong}[r] & \Epi_\leq^\mathrm{tr}(\NN) \ar[r] & \Epi_\leq(\NN)\\
S^\mathrm{op} \ar^-{\cong}[r] \ar_{=}[u] & \Emb_\leq^\mathrm{tr}(\NN) \ar[r]\ar_{-}[u] & \Emb_\leq(\NN)\ar_{-}[u]}
\end{equation*}
commutes.
\end{proposition}

\begin{proof}
To clarify the diagram: horizontal arrows are morphisms, vertical arrows are anti-morphisms given by the identity map, respectively the pseudo-inverse. The fact that the diagram commutes follows from the first item of \Cref{prop:inverse} together with~\eqref{eq:intertwine}.

The only point not yet established is that the representation $S^\mathrm{op}\to \Emb_\leq^\mathrm{tr}(\NN)$ is an isomorphism; the corresponding statement for $S$ was seen in the proof of \Cref{thm:related}.

The injectivity of $S^\mathrm{op}\to \Emb_\leq^\mathrm{tr}(\NN)$ follows from the diagram. For surjectivity, let $f\in \Emb_\leq^\mathrm{tr}(\NN)$; we argue by induction on the translation $k$ uniquely determined by expression~\eqref{eq:trans}. Note that $k\geq 0$ since $f$ is injective and the case $k=0$ corresponds to $f=\mathrm{Id}$. If $f\neq \mathrm{Id}$, we take $a\in\NN$ maximal for the property that $f(n)=n$ for all $n<a$. Thus $f(a)\geq a+1$ and we can write $f = t_a f'$ for some $f'\in \Emb_\leq^\mathrm{tr}(\NN)$ with translation $k-1$; indeed we can take $f'=s_a f$.
\end{proof}

For the record, we note another few properties of the pseudo-inverse.

\begin{proposition}\label{prop:inj-surj}
\ 
\begin{enumerate}
\item $f\mapsto f^-$ is continuous on $\End_\leq^\infty(\NN)$.
\item Let $f, g\in \Epi_\leq(\NN)$. Then $f^- = g^- \Longrightarrow f=g$.
\item For every $f\in \Epi_\leq(\NN)$ there is $g\in \Emb_\leq(\NN)$ with $f=g^-$.
\end{enumerate}
\end{proposition}

\begin{proof}
Continuity follows from the definition. For the injectivity on surjections, suppose for a contradiction $f\neq g$ and let $N\in \NN$ be the smallest integer with $f(N)\neq g(N)$. We can assume $f(N)>g(N)$ and write $a=f(N)$ and $n=f^-(a)=g^-(a)$; note that $n\leq N$. By \Cref{prop:inverse}, $f f^- = g g^- = g f^-$ and therefore the image of $f^-$ must miss $N$. It follows $n<N$ and therefore
\begin{equation*}
g(n) = f(n) \geq a = f(N) > g(N),
\end{equation*}
contradicting that $g$ is order-preserving.

For the last point, the continuity shows that it suffices to justify that $f\in \Epi_\leq^\mathrm{tr}(\NN)$ lies in the image of $\Emb_\leq^\mathrm{tr}(\NN)$; this follows from \Cref{prop:big-picture}.
\end{proof}

Since \Cref{thm:related} describes the amenability properties of the four Polish monoids, we finally record the situation of the four countable sub-monoids (respectively semigroup) encountered along the way.

\begin{corollary}
$\Epi_\leq^\mathrm{tr}(\NN)$, $\Emb_\leq^\mathrm{tr}(\NN)$ and $\End_\leq^\mathrm{tr}(\NN)$ are amenable, whilst $\End_\leq^\mathrm{fin}(\NN)$ is not.
\end{corollary}

\begin{proof}
We have seen that $\Epi_\leq^\mathrm{tr}(\NN)$ is isomorphic to $S$, which is amenable by \Cref{thm:fixed-pt}. As recorded in the proof of \Cref{thm:related}, both $\Emb_\leq^\mathrm{tr}(\NN)$ and $\End_\leq^\mathrm{tr}(\NN)$ have a stronger Schauder--Tychonoff property because they contain $t_1$.

As for $\End_\leq^\mathrm{fin}(\NN)$, it cannot carry an invariant mean because it has distinct left-zero elements, namely the constant maps.
\end{proof}

\subsection*{Epilogue}
We opened by presenting $S$ as a single algebraic object containing all the face maps usually defined separately in every dimension. In closing, we note that the simultaneous realization of $S$ and $S^\mathrm{op}$ inside $\End_\leq(\NN)$ is a natural home for the richer simplicial structure introduced by Eilenberg--Zilber~\cite{Eilenberg-Zilber50}. (We caution that notation and terminology have mutated since~\cite{Eilenberg-Zilber50}, but the objects remain the same.)

Eilenberg--Zilber list the following face and degeneracy relations, where $q$ is the dimension; see~\cite[\S10]{Eilenberg-Zilber50}:
\begin{align*}
\eta_q^i \, \epsilon _q^j &= \epsilon_{q-1}^j \, \eta_{q-1}^{i-1} \kern3mm \forall j<i\\
\eta_q^i \, \epsilon _q^i &= \mathrm{Id}_{q-1} = \eta_q^i \, \epsilon _q^{i+1}\\
\eta_q^i \, \epsilon _q^j &= \epsilon_{q-1}^{j-1}  \, \eta_{q-1}^i \kern3mm \forall i+1<j\\
\eta_{q-1}^i \, \eta_q^j &= \eta_{q-1}^j \, \eta_q^{i+1}  \kern3mm \forall j\leq i\\
\eta_{q-1}^i \, \eta_q^j &= \eta_{q-1}^{j-1} \, \eta_q^i  \kern3mm \forall i <j
\end{align*}
This corresponds to  relations satisfied by $S$ and $S^\mathrm{op}$ when both are embedded in $\End_\leq^\mathrm{tr}(\NN) \subseteq \End_\leq(\NN)$, where all $\eta_q^i$ are replaced by $s_i$ and all $\epsilon _q^i$ by $t_i$. (The only immaterial difference is that face and degeneracy operators are usually indexed by $\NN\cup\{0\}$.) Indeed, the first and third above relations are readily verified for $s_i, t_i$ and all others have already been discussed. In fact, it is not difficult to check that $\End_\leq^\mathrm{tr}(\NN)$ is generated by $S \cup S^\mathrm{op}$; indeed, one can even write $\End_\leq^\mathrm{tr}(\NN) = S  S^\mathrm{op}$.

In conclusion,  the monoid $S$ relevant to ``semi-simplicial topology'' has the fixed point property by \Cref{thm:fixed-pt}, which seems non-trivial. On the other hand, the larger monoid $\End_\leq^\mathrm{tr}(\NN)$ corresponding to ``simplicial topology'' (i.e. with degeneracies) has the stronger Schauder--Tychonoff property simply by virtue of the absorbing nature of the degeneracy $t_1$.


\begin{bibdiv}
\begin{biblist}

\bib{Arens51}{article}{
      author={Arens, Richard},
       title={The adjoint of a bilinear operation},
        date={1951},
        ISSN={0002-9939,1088-6826},
     journal={Proc. Amer. Math. Soc.},
      volume={2},
       pages={839\ndash 848},
}

\bib{Cannon-Floyd-Parry}{article}{
      author={Cannon, James~Welden},
      author={Floyd, William~J.},
      author={Parry, Walter~R.},
       title={Introductory notes on {R}ichard {T}hompson's groups},
        date={1996},
        ISSN={0013-8584},
     journal={Enseign. Math. (2)},
      volume={42},
      number={3-4},
       pages={215\ndash 256},
}

\bib{Carroll1872}{book}{
      author={Carroll, Lewis},
       title={Through the {L}ooking-{G}lass},
   publisher={MacMillan, London},
        date={1872},
}

\bib{Day57}{article}{
      author={Day, Mahlon~Marsh},
       title={Amenable semigroups},
        date={1957},
        ISSN={0019-2082},
     journal={Illinois J. Math.},
      volume={1},
       pages={509\ndash 544},
}

\bib{Day61}{article}{
      author={Day, Mahlon~Marsh},
       title={Fixed-point theorems for compact convex sets},
        date={1961},
     journal={Illinois J. Math.},
      volume={5},
       pages={585\ndash 590},
}

\bib{Dunford-Schwartz_I}{book}{
      author={Dunford, Nelson},
      author={Schwartz, Jacob~Theodore},
       title={Linear {O}perators. {I}. {G}eneral {T}heory},
      series={With the assistance of W. G. Bade and R. G. Bartle. Pure and
  Applied Mathematics, Vol. 7},
   publisher={Interscience Publishers Inc., New York},
        date={1958},
}


\bib{Eilenberg-Zilber50}{article}{
      author={Eilenberg, Samuel},
      author={Zilber, Joseph~Abraham},
       title={Semi-simplicial complexes and singular homology},
        date={1950},
        ISSN={0003-486X},
     journal={Ann. of Math. (2)},
      volume={51},
       pages={499\ndash 513},
}

\bib{FFMN}{unpublished}{
      author={Fournier-Facio, Francesco},
      author={Monod, Nicolas},
      author={Nariman, Sam},
       title={The bounded cohomology of transformation groups of {E}uclidean
  spaces and discs},
        date={2024},
        note={with an appendix by {A.} {Kupers}. Preprint, arXiv:2405.20395v1},
}

\bib{FreyPhD}{thesis}{
      author={Frey, Alexander~Hamilton, Jr},
       title={Studies on amenable semigroups},
        type={Ph.D. Thesis},
        date={1960},
}

\bib{Goldstine38}{article}{
      author={Goldstine, Herman~Heine},
       title={Weakly complete {B}anach spaces},
        date={1938},
        ISSN={0012-7094,1547-7398},
     journal={Duke Math. J.},
      volume={4},
      number={1},
       pages={125\ndash 131},
}

\bib{Huet}{incollection}{
      author={Huet, G{\'e}rard},
       title={Confluent reductions: abstract properties and applications to
  term rewriting systems},
        date={1977},
   booktitle={18th {A}nnual {S}ymposium on {F}oundations of {C}omputer
  {S}cience ({P}rovidence, {R}.{I}., 1977)},
   publisher={IEEE, Long Beach, CA},
       pages={30\ndash 45},
}


\bib{Kakutani38}{article}{
      author={Kakutani, Shizuo},
       title={Two fixed-point theorems concerning bicompact convex sets},
        date={1938},
     journal={Proc. Imp. Acad.},
      volume={14},
      number={7},
       pages={242\ndash 245},
}

\bib{Klawe77}{article}{
      author={Klawe, Maria},
       title={Semidirect product of semigroups in relation to amenability,
  cancellation properties, and strong {F}{\o}lner conditions},
        date={1977},
        ISSN={0030-8730,1945-5844},
     journal={Pacific J. Math.},
      volume={73},
      number={1},
       pages={91\ndash 106},
}

\bib{Moore_hindman}{incollection}{
      author={Moore, Justin~Tatch},
       title={Hindman's theorem, {Ellis}'s lemma, and {Thompson}'s group
  {{\(F\)}}},
        date={2015},
   booktitle={Selected topics in combinatorial analysis},
   publisher={Beograd: Matemati{\v{c}}ki Institut SANU},
       pages={171\ndash 187},
}

\bib{Moore_idem}{article}{
      author={Moore, Justin~Tatch},
       title={Nonexistence of idempotent means on free binary systems},
        date={2019},
        ISSN={0008-4395},
     journal={Can. Math. Bull.},
      volume={62},
      number={3},
       pages={577\ndash 581},
}

\bib{Newman}{article}{
      author={Newman, Maxwell Herman~Alexander},
       title={On theories with a combinatorial definition of ``equivalence''},
        date={1942},
        ISSN={0003-486X},
     journal={Ann. of Math. (2)},
      volume={43},
       pages={223\ndash 243},
}


\bib{vonNeumann29}{article}{
      author={von Neumann, Johann},
       title={Zur allgemeinen {T}heorie des {M}asses},
        date={1929},
     journal={Fund. Math.},
      volume={13},
       pages={73\ndash 116},
}

\end{biblist}
\end{bibdiv}
\end{document}